\documentclass{amsart}[12pt]

  \usepackage[utf8]{inputenc}
  \usepackage{amsmath}
\usepackage{enumerate}
  \usepackage{amssymb}
  \usepackage{amsthm}
  \usepackage{bm}
  \usepackage{graphicx}
  \usepackage{color}
  \usepackage{bigints}
\usepackage{esint}
\usepackage{nccmath}
\usepackage{relsize}
\usepackage{accents}
\usepackage{epigraph}
\usepackage{dsfont}
\usepackage{showlabels}
\usepackage[breaklinks,colorlinks,linkcolor=blue,anchorcolor=blue,citecolor=blue,backref=page]{hyperref}
\usepackage{parskip}

\renewcommand*{\backref}[1]{}
\renewcommand*{\backrefalt}[4]{
	\ifcase #1 (Not cited.)
	\or        (p.\,#2)
	\else      (pp.\,#2)
	\fi}
\usepackage[norefs,nocites]{refcheck}
\usepackage[english]{babel}

\newtheorem{Th}{Theorem}[section]
\newtheorem{lem}[Th]{Lemma}

\theoremstyle{definition}
\newtheorem{Def}[Th]{Definition}

\newtheorem{Cor}[Th]{Corollary}
\newtheorem{Prop}[Th]{Proposition}

\newtheorem{que}{Question}
\newtheorem{Conj}{Conjecture}

\theoremstyle{remark}
\newtheorem{rem}[Th]{Remark}

\newtheorem{thank}{\ \ \ Acknowledgment}

\def\scalar(#1,#2){(#1\mid#2)}

\renewcommand{\hat}{\widehat}

\newcommand{\cb}{\mathcal{B}}

\newcommand{\cm}{\mathcal{M}}

\newcommand{\cl}{\mathcal{L}}

\newcommand{\Pro}{{\mathbb{P}}}
\newcommand{\T}{{\mathbb{T}}}
\newcommand{\Ha}{{\mathbb{H}}}

\newcommand{\Z}{{\mathbb{Z}}}

\newcommand{\N}{{\mathbb{N}}}

\newcommand{\1}{{\mathds{1}}}

\newcommand{\tend}[3][]{\xrightarrow[#2\to#3]{#1}}

\newcommand{\setdef}{\stackrel {\rm {def}}{=}}
\newcommand{\ds}{\displaystyle}

\title[On the Mahler measure of the spectrum of rank one maps]{on the Mahler measure of the spectrum of rank one maps$^{\star}$}
\author{E. H. el Abdalaoui}
\address{Normandie Universit\'e, Universit\'e de Rouen-Math\'ematiques, \\
  Labo. de Maths Raphael SALEM  UMR 60 85 CNRS
Avenue de l'Universit\'e, BP.12
76801 Saint Etienne du Rouvray - France }
\email{elhoucein.elabdalaoui@univ-rouen.fr}
\thanks{}
\dedicatory{$^{\star}$ Dedicated to Professor Fran\c{c}ois Parreau}
\date{June 20, 2013}
\begin{document}
\begin{abstract}
We extend partially the Kakutani-Zygmund dichotomy theorem to a class of generalized Riesz-product type measures by proving that the generalized Riesz-product is singular if and only if its Mahler measure is zero. As a consequence, we exhibit a new subclass of rank one maps acting on a finite measure space with singular spectrum. In our proof the $H^p$ theory coming to play. Furthermore, by appealing to a deep result of Bourgain, we prove that the Mahler measure of the spectrum of rank one map with cutting parameter $p_n=O(n^\beta)$, $\beta \leq 1$ is zero, and  we establish that the integral of the square root of the absolute part of any generalized Riesz-product is strictly less than 1. This answer partially a question asked by M. Nadkarni.  
\end{abstract} 
\maketitle
\epigraph{What matters to an active man is to do the right thing; whether the right thing comes to pass should not bother him.}{\textit{ Goethe }}

\epigraph{Fej\'{e}r used to say-in the 1930's, "Everybody writes and nobody reads." This was true eventhen. Reviewing has improved, but even so it is very hard.} {\textit{ Erd\"{o}s}}
\section{Introduction}
The purpose of this paper is to investigate the properties of the Mahler measure of the maximal spectral type of rank one maps and to discuss the question on whether the Mahler measure of the  spectrum of rank one maps acting on a finite measure space is zero. According to our first main result, this is question concern the problem of the singularity of the spectrum of rank one maps acting on finite measure space \cite[pp.14]{Frenczi}, \cite{Thouvenot}. Rank one
maps have simple spectrum and  using a random procedure, D. S. Ornstein
produced a family of mixing
rank one maps \cite{Ornstein}. It follows that Ornstein's class of
maps may possibly contain a candidate for Banach-Rhoklin's 
problem whether there exists a map $T$ acting on the probability space  $(\Omega,{\mathcal
{A}},\Pro,)$  with simple Lebesgue spectrum \footnote{For more details on the well-known Banach's problem, see the next paragraph entitled Banach's problem. }. Unfortunately, in 1993, J.
Bourgain proved that almost surely Ornstein's
maps have singular spectrum \cite{Bourisr}. Subsequently, using the
same methods, I. Klemes \cite{Klemes1}
showed that the  subclass of staircase maps has singular maximal spectral type. In particular, this subclass contains the
mixing staircase maps of Adams-Smorodinsky \cite{Adams}. Using a refinement of Peyri\`ere criterium \cite{Peyriere}, I. Klemes \& K. Reinhold proved that the rank one maps have a singular spectrum if the inverse of the cutting parameter is not in $\ell^2$ (that is, $\sum_{k=1}^{+\infty} \frac1{m_k^2}=+\infty$, where $(m_k) \subset \Big\{2,3,4,\cdots\Big\}$  is the cutting parameter) \cite{Klemes2}. This class contains the mixing staircase maps of Adams \&
Friedman \cite{Adams2}. In 1996, H. Dooley and S. Eigen adapted the Brown-Moran methods \cite[pp.203-209]{GrahamMc} and proved that the spectrum of a subclass of Ornstein maps is almost surely singular \cite{DooleyE}.\\

Later, el Abdalaoui-Parreau and Prikhod'ko extended Bourgain theorem \cite{Bourisr} by proving that for any family of probability measures in Ornstein type constructions, the corresponding maps have
almost surely a singular spectrum \cite{elabdalihp}. They obtained the same result for Rudolph's construction \cite{Rudolph}.
In 2007, el Abdalaoui showed that the spectrum of the rank one map is singular
provided that the spacers  $(a_j)_{j=1}^{m_k} \subset \N$,
are lacunary for all $k$ \cite{elabdaletds}. The author used the Salem-Zygmund Central Limit Theorem methods. As a consequence, the author presented a simple proof of Bourgain theorem \cite{Bourisr}.

Five years later, by appealing to the martingale approximation technique, C. Aistleitner and M. Hofer \cite{AH} proved a counterpart of the result in \cite{elabdaletds}. Precisely, they proved that the spectrum of the rank one maps is singular provided that the cutting parameter  $(m_k) \in \N^*$ and the spacers  $(a_j)_{j=1}^{m_k} \subset \N$ satisfy:
\begin{enumerate}[i)]
\item $\ds \frac{\log(m_{k_n})}{h_{k_n}}$ converge to $0$;
\item the proportion of equal terms in the spacers is at least $c.m_{k_n}$ for some fixed constant $c$ and some subsequence $(k_n)$.\\
\end{enumerate}

Seven years ago \cite{elabdal-Nad}, e. H. el Abdalaoui and M. Nadkarni produced an infinite product formula for the Mahler measure of the absolutely continuous part of the generalized Riesz product related to the Mahler measure of the polynomials $(P_k)_{k \in \N}$ associated to it. By appealing to the $H^p$ theory, they proved that this infinite product is exactly the Mahler measure of its . They further established that the  infinite product of $(P_k)_{k \in \N}$ converge in $L^1$ to the square root of the absolutely continuous part of the generalized Riesz product.  From this, they conjectured that the logarithm of the absolutely continuous part of the spectrum of any rank one map acting on a space of \emph{finite} measure is not integrable. Moreover, they proved that the Banach's problem, Littlewood problem and Mahler problem are equivalent in the class of rank one maps acting on a space of infinite measure.\\

\paragraph{\textbf{Banach's problem.}}Following Ulam \cite[p.76]{Ulam}, Banach asked:
\begin{que}[Banach Problem]
	Does there exist a square integrable function $f(x)$ and a measure preserving transformation $T(x)$,
	$-\infty<x<\infty$, such that the sequence of functions $\{f(T^n(x)); n=1,2,3,\cdots\}$ forms a complete
	orthogonal set in Hilbert space?\footnote{Professor M. Nadkarni pointed to me that the question contain an oversight. The  sequence of functions should be bilateral, that is, $n \in \Z$.}
\end{que}
Obviously, as pointed by Rokhlin, the map  $n \in \Z \mapsto n+1$ with the indicator function of $\{1\}$ answer the question. But the map is dissipative. For the conservative case, el Abdalaoui obtained also an affirmative answer by producing an infinite rank one map \cite{AS}. The proof is based on the ideas from \cite{elabdal-Nad} \cite{Abd-Nad2} combined with the singer's theoretical number theory theorem and tools from extremal harmonic analysis. The strategy is accomplished, as in \cite{Abd-Nad1}, by producing a sequence of flat analytic polynomials with coefficients $0,1$. But the polynomials are only flat in $L^1$-sense instead  in \cite{Abd-Nad1}, the polynomials are ultra-flats with  positive coefficients.
  
The most famous Banach problem in ergodic theory (which should be attributed to Banach and Rokhlin) asks if there is a measure preserving transformation on a probability space which
has simple Lebesgue spectrum. A similar problem is mentioned by Rokhlin in \cite[p.219]{Rokh}. Precisely, Rokhlin asked on the existence
of an ergodic measure preserving transformation on a finite measure space whose spectrum is Lebesgue type with finite multiplicity.
Later,  Kirillov in his 1966's paper \cite{Kiri} wrote ``there are grounds for thinking that such examples do not exist".
However he has described a measure preserving  action (due to M. Novodvorskii) of
the group $(\bigoplus_{j=1}^\infty{\mathbb {Z}})\times\{-1,1\}$ 
on the compact dual of discrete rationals whose unitary group has Haar spectrum of multiplicity 2. Similar group actions with higher finite even multiplicities are also given.\\

Subsequently, finite measure preserving transformation having Lebesgue component of finite even multiplicity have been
constructed by J. Mathew and M. G. Nadkarni \cite{MN}, Kamae \cite{Kamae}, M. Queffelec \cite{Queffelec1}, and  O. Ageev \cite{Ag}.
Fifteen years later, M. Guenais \cite{Guenais} used a $L^1$-flat generalized Fekete polynomials on some torsion groups 
to construct a group action with simple Lebesgue component. 
A straightforward application of Gauss formula yields that the generalized Fekete polynomials constructed by Guenais
are ultraflat. Very recently, el Abdalaoui and Nadkarni strengthened Guenais's result \cite{Abd-Nad1} by proving that there exist an ergodic non-singular dynamical 
system with simple Lebesgue component. However, despite all these efforts, it is seems that the question of Rokhlin still open since the maps constructed does not have {\it a pure
	Lebesgue spectrum.}
\\ 

Here, we are able to extend partially Kakutani-Zygmund dichotomy theorem by proving that the maximal spectral type of the rank one map is either singular or its absolute continuous part is equivalent to the Lebesgue measure, depending whether the product of Mahler measure of $(P_k)_{k\in \N}$ converges or diverges (Theorem \ref{dicho} below).

This combined with the recent result of el Abdalaoui and Nadkarni in \cite{elabdal-Nad} and \cite{Abd-Nad2} allows us to establish that the rank one maps has a singular spectrum provided that each $P_k$ has less than $c.h_{k-1}$ zeros bigger than 1 in absolute value
where $c$ is a positive constant less than one.\\

Furthermore, we are able to give a partial negative answer to the question of \linebreak Nadkarni by proving that the integral of the square root of the absolute part of any non trivial generalized Riesz products is strictly less than 1. Furthermore, using a deep result of Bourgain \cite{Bourisr}, we establish that there is a new subclass of rank one maps acting on \emph{finite and infinite} measure space for which the Mahler measure is zero. For this subclass the cutting parameter satisfies $m_k=\theta(k^{\beta})$, for some $\beta \leq 1$. As a consequence, we obtain that the spectrum of any map in this class is singular. However, we are not able to answer the question on whether the Mahler measure on any rank one acting on probability space is zero. Nevertheless, we made the following conjecture.

\begin{Conj}Let $(X,\cb,\Pro,T)$ be a rank one map and assume that
	$\Pro(X)<+\infty$. Then, the Mahler measure of its spectrum is zero.
\end{Conj}



Obviously, this conjecture is related to Klemes-Reinhold's conjecture \cite{Klemes2} on the spectrum of rank one. Indeed, therein the authors conjectured that 
all rank one maps have singular spectrum and, in the same spirit,
C. Aistleitner and M. Hofer wrote in the end of their paper ``several authors believe that all rank one transformations could have singular maximal spectral type.'' \cite{AH}. It seems that this conjecture was formulated since Baxter result \cite{Baxter}, \cite{Thouvenot}. Of-course, the conjecture  stand only for rank one acting on finite measure space.

The paper is organized as follows. In section \ref{rkone} we recall the definition of rank one map and some standard facts on the notion of Mahler measure. This allows us to state our main results in section \ref{S-mains} . In section \ref{S-tool}, we review some basic facts on the affinity and the mutual singularity of two probability measures. Finally, in section \ref{S-Proofs}, we present the proofs our ours main results.\\


\section{rank one maps by cutting and stacking methods}\label{rkone}
Using the cutting and stacking method described in \cite{Friedman1}, \cite{Friedman2},
one can construct inductively a family of measure preserving
maps, called rank one maps, as follows
\vskip 0.1cm Let $B_0$ be the unit interval equipped with
Lebesgue measure. At stage one we divide $B_0$ into $m_0$ equal
parts, add spacers and form a stack of height $h_{1}$ in the usual
fashion. At the $k^{th}$ stage we divide the stack obtained at the
$(k-1)^{th}$ stage into $m_{k-1}$ equal columns, add spacers and
obtain a new stack of height $h_{k}$. If during the $k^{th}$ stage
of our construction  the number of spacers put above the $j^{th}$
column of the $(k-1)^{th}$ stack is $a^{(k-1)}_{j}$, $ 0 \leq
a^{(k-1)}_{j} < \infty$,  $1\leq j \leq m_{k}$, then we have
$$h_{k} = m_{k-1}h_{k-1} +  \sum_{j=1}^{m_{k-1}}a_{j}^{(k-1)}.$$
\begin{figure}[hbp]
\begin{center}
\scalebox{0.5}{\input{rankone.pstex_t} }
\end{center}
\end{figure}
\noindent{}Proceeding in this way we get a rank one map
$T$ on a certain measure space $(X,{\mathcal B},\mid.\mid)$ which may
be finite or
$\sigma-$finite depending on the number of spacers added. \\
\noindent{} The construction of a rank one map thus
needs two parameters, $(m_k)_{k=0}^\infty$ (cutting parameter), and $((a_j^{(k)})_{j=1}^{m_k})_{k=0}^\infty$
(spacers parameter). Put

$$T \stackrel {def}= T_{(m_k, (a_j^{(k)})_{j=1}^{m_k})_{k=0}^\infty}$$

\noindent In \cite{Nadkarni1} and \cite{Klemes2} it is proved that
the spectral type of this map is given (up to possibly some discrete measure) by

\begin{eqnarray}
d\mu  ={\rm{W}}^{*}\lim \prod_{k=1}^n\big| P_k\big|^2d\lambda,
\end{eqnarray}
\noindent{}where
\begin{eqnarray*}
&&P_k(z)=\frac 1{\sqrt{m_k}}\left(1+
\sum_{j=1}^{m_k-1}z^{-(jh_k+\sum_{i=1}^ja_i^{(k)})}\right),\nonumber  \\
\nonumber
\end{eqnarray*}
\noindent{}$\lambda$ denotes the normalized Lebesgue measure on the
circle group $\T$ and $\rm{W}^{*} \lim$ denotes weak$*$limit in the space of
bounded Borel measures on ${\T}$.

\noindent{}As mentioned by Nadkarni in \cite{Nadkarnibook}, the infinite product
$$
\prod_{l=1}^{+\infty}\big|P_{j_l}\big(z)|^2$$
\noindent{}taken over a subsequence $j_1<j_2<j_3<\cdots,$ also represents the maximal spectral type (up to discrete measure) of some rank one maps. In case $j_l \neq l$ for infinitely many $l$, the maps acts on an infinite measure space.\\

We will denote by $\xi_n(x)\setdef{1\over{\sqrt{|B_n|}}} \1_{B_n}(x)$ the
characteristic function of the $n^{\hbox{th}}$--base,
 normalized
so that the 2-norm equals 1. We recall that the associated Koopman operator  $U_T \xi$  is define by
$U_T \xi(x)=\xi(T^{-1}x)$ on $L^2(X)$ and, for any $\xi \in L^2(X)$ there corresponds a positive measure
$\mu_\xi$ on $\T$, the unit circle, defined by
$\hat{\mu}_\xi(n)= \langle U_T^n\xi,\xi \rangle$. With this notation, put  $\mu_n=\mu_{\xi_n}$. Notice that
$${\mathcal C}=\{\{T^k(B_n)\}_{k=0}^{h_n-1}\}_{n=0}^{\infty} \eqno{(1.2)}$$
generates a dense  subalgebra of the Borel $\sigma$--algebra,
(here we are using the metric (modulo sets of measure zero) given by
$d(A,B)=\big|A\triangle B\big|$). Then
the subspace generated by  the span of
$\{U_T^k(\xi_n):1\leq n<\infty,0\leq k<h_n \}= \hbox{ span of
}\{\1_{T^k(B_n)}:1\leq n<\infty,0\leq k<h_n \}$
is dense in  $L^2(X)$.

We end this section by introducing the notion of Mahler measure and stating our second main result. The Mahler measure of $P_k$ is defined by
\[
M(P_k)=\exp \Biggl(\bigintss_{\T} \log\big(\big|P_k(z)\big|\big) dz \Biggr).
\]
Using Jensen's formula \cite{Rudin}, it can be shown that
\[
M(P_k)=\frac1{\sqrt{m_k}}\prod_{|\alpha|>1} |\alpha|,
\]
where, $\alpha$ denoted the zero of the polynomial $\sqrt{m_k}P_k$. In this definition, an empty product is assumed to be $1$ so the Mahler measure of the non-zero constant polynomial $P(x)=a$ is $|a|$. A nice account on the subject may be founded in \cite[pp.2-11]{Ward}, \cite{Borwein}.\\

\noindent{}Here are some elementary properties of  the Mahler measure. But, we provide a proof for the reader's convenience.
\begin{Prop}\label{basic}Let $(X,\cb,\rho)$ be a probability space. Then,
for any two positive functions $f,g \in L^1(X,\rho)$, we have
\begin{enumerate}[i)]
\item $M(f)$ is a limit of the norms $||f||_{\delta}$ as $\delta$ goes to $0$, that is,  \[
||f||_{\delta} \setdef \Biggl(\bigintss f^{\delta} d\rho\Biggr)^{\frac1{\delta}} \tend{\delta}{0} M(f),
\]
provided that $\log(f)$ is integrable.
\item If $\rho\Big\{ f >0 \Big\} <1$ then $M(f)=0$.
\item If $0<p<q<1$, then $\bigl\|f\bigr\|_p \leq \bigl\|f\bigr\|_q$.
\item If $0<p< 1$, then $M(f) \leq \bigl\|f\bigr\|_p$.
\item $\ds \lim_{\delta \longrightarrow 0}\int f^\delta d\rho = \rho\Big\{f>0\Big\}.$
\item $M(f) \leq \bigl\|f\bigr\|_1$.
\item $M(fg)=M(f)M(g)$.
\end{enumerate}
\end{Prop}
\begin{proof}We start by proving ii). Without loss of generality, assume that $f \geq 0$ and put
 $$B=\Big\{ f >0 \Big\},$$
and let $\delta=1/k$  be in $]0,1[$, $k \in \N^*$. Then $1/(1/\delta)+1/(1-\delta)=1/k+(k-1)/k=1$. Hence, by H\"{o}lder inequality, we have
\begin{eqnarray*}
\bigintss f^{\delta} d\rho &=& \bigintss f^{1/k} .\1_B d\rho \\
&\leq& \Biggl(\bigintss (f^{1/k})^k dz\Biggr)^{1/k} \Biggl(\bigintss \1_B^{k/k-1} dz\Biggr)^{k-1/k}\\
&\leq& \Biggl(\bigintss f d\rho\Biggr)^{1/k} \Biggl(\bigintss \1_B dz\Biggr)^{k-1/k}\\
&\leq& \Biggl(\bigintss f d\rho\Biggr)^{1/k} \Bigl(\rho(B)\Bigr)^{(k-1)/k}
\end{eqnarray*}
Therefore we have proved
\begin{eqnarray*}
||f||_{\delta} &\leq& \Biggl(\bigintss f d\rho \Biggr) \Bigl(\rho(B)\Bigr)^{(1-\delta)/\delta}\\
&\leq& \Biggl(\bigintss f d\rho \Biggr) \Big(\rho(B)\Big)^{k-1} \tend{k}{+\infty}0,
\end{eqnarray*}
To prove i), apply the Mean Value Theorem to the following functions
\[\left\{
  \begin{array}{ll}
    \delta \longmapsto x^\delta, & \hbox{if $x \in ]0,1[;$}  \\
    t \longmapsto t^\delta , & \hbox{if $x>1,$}
  \end{array}
\right.
\]
Hence, for any $\delta \in ]0,1[$ and for any $x>0$, we have
\[
\Biggl|\frac{x^\delta-1}{\delta}\Biggr| \leq x+\Bigl|\log(x)\Bigr|.
\]
Furthermore, it is easy to see that
\[
\frac{f^\delta-1}{\delta}=\frac{e^{\delta \log(f)}-1}{\delta}
\tend{\delta}{0}\log(f),
\]
and, by Lebesgue Dominated Convergence Theorem, we get that
\[
\bigintss\frac{f^\delta-1}{\delta} d\rho \tend{\delta}{0} \bigintss \log(f) d\rho.
\]
On the other hand, for any $\delta \in ]0,1[$
\[
\bigl|\bigl|f\bigr|\bigr|_{\delta}=\exp\Biggl({\frac1{\delta}}\log\Biggl(\bigintss f^{\delta} d\rho\Biggr)\Biggr),
\]
and for a sufficiently small $\delta$, we can write
\[
{\frac1{\delta}}\log\Biggl(\bigintss f^{\delta} d\rho\Biggr)\sim
\bigintss \frac{f^\delta-1}{\delta} d\rho
\]
since $\log(x) \sim x-1$ as $x \longrightarrow 1$.
Summarizing we have proved
\[
\lim_{\delta \longrightarrow 0}||f||_{\delta}=exp\Biggl(\bigintss \log(f) d\rho\Biggr)=M(f).
\]
For the proof of iii) and iv), notice that the function $x \mapsto \ds x^\frac{q}{p}$ is a convex function and
$x \mapsto \log(x)$ is a concave function. Applying Jensen's inequality to $\ds \bigintss \bigl|f\bigr|^p d\rho$ we get
$$\bigl\|f\bigr\|_p \leq \bigl\|f\bigr\|_q,~~~~~~~~~
\bigintss \log(\bigl|f\bigr|) d\rho \leq \log\Bigl(\bigr\|f\bigl\|_p \Bigr),$$
and this finishes the proof, the rest of the proof is left to the reader.
\end{proof}
Szeg\"{o}-Kolmogorov-Krein theorem established a connection between a given measure and the Mahler measure of its derivative. Precisely, we have
\begin{Th}[Szeg\"{o}, Kolmogorov-Krein {\cite[p.49]{Hoffman}, \cite[p.136]{Gamelin}}.]\label{Szego}Let $\sigma$ be a finite positive Baire measure on the unit circle and let $h$ be the derivative of $\sigma$ with respect to normalized Lebesgue measure. Then, for any $r>0$,
\[
M(h)=\inf_{P}\Big\|1-P\Big\|_r^{r}=\inf_{P}\Biggl(\bigintss \Big|1-P\Big|^r h(z) dz\Biggr),
\]
where $P$ ranges over all analytic trigonometric polynomials with zero constant term. The right side is $0$ if $\log(h)$ is not integrable.
\end{Th}
Clearly, Szeg\"{o}-Kolmogorov-Krein theorem gives an alternative definition to Mahler measure (that is, the Malher measure
of a given measure is the Mahler measure of its derivative). For other definitions, we refer the reader to \cite{degot}.\\

The Mahler measure is very useful in number theory and the use of this quantity in number theory is essentially due to Mahler \cite{Mahlerintro}.\\

Following Helson and Szeg\"{o} \cite{HS}, Szeg\"{o}-Kolmogorov-Krein theorem solved the first problem of the theory of prediction and the second problem of this theory was solved by Kolmogorov as follows

\begin{Th}[Kolmogorov {\cite[p.49]{Grenard}}.]
Let $\sigma$ be a finite positive Baire measure on the unit circle and let $h$ be the derivative of $\sigma$ with respect to normalized Lebesgue measure. Then,
\[
M_0(h)=\inf_{P \in \cm_0}\Big\|1-P\Big\|^2_2=\Biggl(\bigintss \frac1{h} dz \Biggr)^{-1},
\]
where $\cm_0$ is a closed subspace generated by $\{z^n, n \in \Z \setminus\{0\} \}$. The right side is $0$ if $\log(h)$ is not integrable.
\end{Th}
In his 1984's paper {\cite{Nakazi}, Nakazi extended simultaneously Kolmogorov prediction theorem and the following result due to Nakazi and Takahashi \cite{Nakazi-Takahashi}}
\begin{Th}[Nakazi-Takahashi{\cite{Nakazi-Takahashi}, \cite{ahamid}}.]
Let $\sigma$ be a finite positive Baire measure on the unit circle and let $h$ be the derivative of $\sigma$ with respect to normalized Lebesgue measure. Then,
\[
M_n(h)=\inf_{P \in \cm_n}\Big\|1-P\Big\|_2=\Bigl(\sum_{k=0}^{n}|\alpha_k|^2\Bigr)^{\frac12},
\]
where $\cm_n$ is a closed subspace generated by $\{z^k, k \geq (n+1) \}$ and $(\alpha_k)_{k=0}^{+\infty}$ is
the Fourier coefficients of the associated outer function $\phi$ to $h$ (that is, $h=|\phi|^2=
\big|\sum_{k=0}^{+\infty}\alpha_k z^k|^2$). The right side is $0$ if $\log(h)$ is not integrable.
\end{Th}

Nakazi Theorem generated considerable interest in computing the predicator error when the index set is $\{1,2,\cdots\}$  with finitely many points of $\Z$ added or deleted. For a recent results we refer the reader to \cite{ahamid}, \cite{ahamid2} and the references therein.

Let us mention that the fundamental ingredients in the proofs of the previews results are based on the Hardy space Theory.
For $0<p<\infty$, the Hardy space $\Ha^p$ is the $L^p(dz)$-closure  of $\{1\}+P_0$, where $P_0$ is the manifold of trigonometric polynomials whose frequencies are in $\big\{1,2,3,\cdots,\big\}$. $\Ha^{\infty}$ is defined to be the weak star closure of  $\{1\}+P_0$ in $L^{\infty}(dz)$.\\

We further have the following Zygmund's theorem \cite[p.96]{koosis}

\begin{Th}[$L log L$ Zygmund's theorem.]\label{zyg} Let $|f(\theta)| \log^{+}(|f(\theta)) \in L^1(-\pi,\pi)$. Then
	$$F(z)=\bigintss \frac{e^{it}+z}{e^{it}-z} f(t) dt $$ 
belonging to $H^1$.
\end{Th}

Based on Szeg\"{o}-Kolmogorov-Krein theorem (Theorem \ref{Szego}) and its allies,  el Abdalaoui-Nadkarni proved the following theorem \cite{elabdal-Nad} that we need here. 

\begin{Th}\label{elabdal-Nad1}Let $T$ be a rank one maps with the cutting parameter $(m_k)$ and spacers parameter $((a_j^{(k)})_{j=1}^{m_k})_{k=0}^\infty$ and let $\mu=\prod_{j \geq 0} \big|P_j\big|^2$ be its spectral type. Then
$$M\Biggl(\frac{d\mu}{dz}\Biggr)=\prod_{j=0}^{+\infty}M\big(P_j^2\big)=\lim_{n\rightarrow \infty}\exp\Biggl\{\bigintss_{S^1}\log \bigl(\prod_{j=1}^n |P_j(z)|^2 \bigr)dz\Biggr\},$$
\noindent{}where $\ds \frac{d\mu}{dz}$ is a derivative of $\mu$.
\end{Th}
We also need the following theorem from \cite{elabdal-Nad}.

\begin{Th}\label{affinity}Let $\mu$ be a generalized Riesz product based on the sequences of polynomials $P_k$, that is,
	\[
	\mu=\prod_{k=0}^{+\infty} |P_k(z)|^2\setdef W*\lim\prod_{k=0}^{N} |P_k(z)|^2 dz,
	\]
	where
	$$P_k(z) =\frac1{\sqrt{m_k}}\sum_{j=0}^{m_k-1}z^{jh_k+s(k,j)},~~s(k,0)=0 {~{and}}~ s(k,j)=\sum_{i=1}^{j}a_i^{(k)}.$$
	Then the product $Q_N=\ds \prod_{k=1}^{N}|P_k(z)|$ converge in $L^1(dz)$ to $\ds\sqrt{\frac{d\mu}{dz}}$.
\end{Th}
\noindent{}We shall also need the following Proposition \cite{elabdal-Nad}.
\begin{Prop}\label{weakly}The sequence $\ds \Big(\prod_{k=0}^{n}|P_k(z)|\Big)$ converge weakly in $L^2(dz)$ to
	$\ds\sqrt{\frac{d\mu}{dz}}$.
\end{Prop}

Let us point out that the rank one maps arising from the generalized Riesz product $\ds \mu=\prod_{k \geq 0} |P_k|^2$ and for any subsequence $\mathcal{N} \subset N$, let us denote by $\nu$ the generalized Riesz product over the subsequence
$\mathcal{N}$, that is, the measure obtain as weak limit of the sequence of measures
$$\ds\prod_{\overset{k \in \mathcal{N}}{k \leq n}} |P_k(z)|^2 dz.$$
Given two subsequences $\mathcal{N}_1$ and $\mathcal{N}_2$ we construct three generalized Riesz product as follows
$$\mu_1=\ds\prod_{k \in \mathcal{N}_1,} |P_k|^2, \mu_2=\ds\prod_{k \in \mathcal{N}_2,} |P_k|^2
{\textrm {~~and~~}}  \nu=\ds\prod_{k \in \mathcal{N}_1 \cup \mathcal{N}_2 } |P_k|^2.$$
We shall establish some relation between the absolutely continuous part of all those three measures. Indeed, we have
\begin{lem}\label{product} $\ds \sqrt{d\nu/dz} = \sqrt{d\mu_1/dz} \sqrt{d\mu_2/dz}.$
\end{lem}
\begin{proof}For any $n\in \N$, put
	$$Q_n=\prod_{\overset{k \in \mathcal{N}_1}{k \leq n}} |P_k| {\textrm{~and~}}
	R_n=\prod_{\overset{k \in \mathcal{N}_2}{k \leq n}} |P_k|.$$
	Then, by Theorem \ref{affinity}, $Q_nR_n$ converge to $\sqrt{d\mu/dz}$ almost everywhere over a subsequence and the same holds for the subproducts $Q_n$ and $R_n$. Thus, over a common subsequence for all three cases the convergence holds almost everywhere. Whence
	$$\ds \sqrt{d\nu/dz} = \sqrt{d\mu_1/dz} \sqrt{d\mu_2/dz},$$
	which proves the lemma.
\end{proof}

\noindent{}We have also the following fundamental lemma.
\begin{lem}
	There is a subsequence $\Big(\ds \frac1{Q_{n_i}}\Big)_{n \geq 0}$ and $\xi$ in $L^2(\mu)$ such that
	$\Big(\ds \frac1{Q_{n_i}}\Big)_{n \geq 0}$ converge weakly to $\xi$ in $L^2(\mu)$.
\end{lem}
\begin{proof} Observe that we have 
	\begin{eqnarray}
	\mu= Q_n^2 d\mu_n, \; \; \textrm{with} \; \; \mu_n=\prod_{ k=n}^{+\infty}|P_k(z)|^2.
	\end{eqnarray}
	 This combined with Cauchy-Schwarz inequality yields
	\[
	\bigintss  {\frac1{Q_{n}} d\mu \leq \Biggl(\bigintss {\frac1{Q_{n}^2}}} d\mu\Biggr)^{\frac12}=
	\Biggl(\bigintss d\mu_n\Biggr)^{\frac12}=1.
	\]
	Hence, the sequence $\Big(\ds \frac1{Q_{n}}\Big)_{n \geq 0}$ is bounded in $L^2(d\mu)$ and this implies that there is a subsequence that converges weakly to some function $\xi$ in $L^2(\mu)$. The proof of the lemma is complete.
\end{proof}
\section{Main results}\label{S-mains}
We start by stating our  first main result.

\begin{Th}[First main result]\label{m1} Let $(X,\cb,\Pro,T)$ be a rank one map. Then, the spectral type is either singular or its absolute continuous part is Lebesgue.
\end{Th}

For our second main result, by appealing to Theorem \ref{elabdal-Nad1}, we are able to extended Kakutani theorem and to prove the analogous of Zygmund dichotomy theorem. Precisely, we have
\begin{Th}[Second main result]\label{dicho} The spectral type of any rank one maps is either singular or its absolute continuous part is equivalent to Lebesgue measure with positive Mahler measure, that is,
    \[
    \lambda \sim \frac{d\mu}{d\lambda}(\theta) d\lambda {~or~} \mu \perp \lambda {\rm {~~according~~as}} \prod_{j \geq 0} M(P_j) {\rm {~converge~or~diverge.}}
    \]
    \end{Th}
Our third main result gives a partial answer to the question raised by M. Nadkarni. We state it as follows
\begin{Th}[Third main result]\label{nad3} Let   $\nu=\prod_{j=0}^{+\infty} \bigl|P_{n_j}\bigr|^2$ be a generalized Riesz product. Then 
$$\bigintss \sqrt{\frac{d\nu}{dz}} dz <1.$$
\end{Th}
Notice that the spectral measure of any indicator function in any dynamical system with phase space of finite measure has a point mass at $1$. Therefore, for any rank one map acting on finite space its spectral type satisfy 
$$\bigintss \frac{d\mu}{dz} dz \leq 1-\frac{1}{|X|},$$
since $\mu $ has a mass $1/ |X|$ at $1$. 

The proof of Theorem \ref{nad3} combined with the deep result of Bourgain \cite{Bourisr}, allows to obtain the following.
\begin{Th}[fourth main result ]\label{main2}For any rank one maps with the cutting parameter $(m_j)$ satisfying
$m_j=\theta(j^\beta)$, for some $\beta \leq 1$, we have
$$\prod_{j=0}^{+\infty}M\big(P_j\big)=0.$$
\end{Th}
\noindent{}Let us point out that our methods combined with Bourgain proposition \cite{Bourisr}, allows us to deduce the following.
\begin{Cor}\label{singular}
Any rank one maps with the cutting parameter $(m_j)$ satisfying
$m_j=\theta(j^\beta)$, for some $\beta \leq 1$ have a singular spectrum.
\end{Cor}
We recall that in the standard asymptotic notation called Bachmann-Landau notation \cite{Flajolet}: $m_j=\theta(j^\beta)$ means that there exits $c,C>0$ and a large $j_0$ such that for any $j \geq j_0$, we have
$c \leq \ds \frac{m_j}{j^\beta} \leq C.$

\noindent{}Furthermore, let us mention that  el Abdalaoui and Nadkarni  proved:
\begin{Th}[\cite{elabdal-Nad}]\label{AN}If each $P_k$ has less than $c.h_{k-1}$ zeros bigger than 1 in absolute value
where $c$ is a positive constant less than one, then $\ds M\Bigl(\frac{d\mu}{dz}\Bigr)=0.$
\end{Th}
As a consequence, we have
\begin{Cor}\label{CAN}If each $P_k$ has less than $c.h_{k-1}$ zeros bigger than 1 in absolute value
where $c$ is a positive constant less than one, then the associated rank one map has a singular spectrum.
\end{Cor}

\begin{rem}Almost all know results on the singularity of rank one maps can be derived from our method combined with our first main result.
\end{rem}
\section{Affinity between two measures}\label{S-tool}

The affinity or Hellinger integral between two finite measures is defined by the integral of the corresponding geometric mean. It was introduced and studied in a series of papers by K. Matusita
\cite{Matusita1},\cite{Matusita2},\cite{Matusita3} and  it is also called  Bahattacharyya coefficient \cite{Bhatta}.
 The affinity between two probability measures $\sigma$ and $\rho$ is defined by
\begin{equation}
G(\sigma,\rho)=\ds \bigintss \sqrt{\frac{d\sigma}{d\tau}. \frac{d\rho}{d\tau}} d\tau.
\end{equation}
This definition does not depend on $\tau$.
The affinity is related to the Hellinger distance as it can be defined as
\[
H(\sigma,\rho)=\sqrt{2(1-G(\sigma,\rho))}.
\]
Note that $G (\sigma ,\rho )$ satisfies (by Cauchy-Schwarz inequality)
\[
 0 \leq G(\sigma,\rho) \leq  1.
\]
It is an easy exercise to see that $G(\sigma,\rho)=0$ if and only if $\sigma$ and $\rho$ are mutually singular
(denoted by $\sigma \bot \rho$.) and $G(\sigma,\rho)=1$ holds if and only if
$\sigma$ and $\rho$ are equivalent.\\

Using the affinity, T. Kamae in \cite{Kamae} and Coquet-Mand\'es-France-Kamae in \cite{Coquet-France} introduced a tools to study the spectral proprieties of $q$-multiplicative sequences. They proved the following result.
\begin{Th}[Coquet-Mand\`es-France-Kamae  \cite{Coquet-France}]\label{coquet-france}Let $(\sigma_n)$ and $(\rho_n)$ be two sequences of probability measures on the circle weakly converging to the probability measures $\sigma$ and $\rho$ respectively. Then
\begin{equation}
\limsup_{n \longrightarrow +\infty } G(\sigma_n,\rho_n) \leq G(\sigma,\rho),
 \end{equation}
\end{Th}
Let us mention that the affinity methods can be used to establish the celebrated Kakutani theorem and Hajek-Feldmen theorem \cite{Prato}.\\

G. Ritter \cite{Ritter} and Brown-Moran  \cite{Brown-Moran} use the same methods in the context of the classical Riesz products. They mentioned that the dissociation can be viewed as a analogous of the stochastic independence. Thus, in this context, the analogous of Kakutani and Hajek-Feldmen theorems is known as Zygmund dichotomy theorem \cite[pp.263-264]{Zygmund}\\

 Later, using the affinity methods combined with the Bourgain tools, el Abdalaoui \cite{elabdalisr} established that almost surely the spectral types of Ornstein maps are mutually singular.\\

Here, we are able to obtain a refinement of the Coquet-Mand\'es-France-Kamae theorem (Theorem \ref{coquet-france}) by proving that the sequence $\prod_{k\geq 0}|P_k|$ converge in $L^1$ to the square root of the derivative of $\mu$. It is turn out that this result is a strong ingredient in the proof of our main results.
We shall need also the following lemma inspired by Bourgain and Kilmer-Seaki methods.
\begin{lem}\label{orth}
Let $\rho$ be a probability measure on the circle $\T$ and $\sigma_n=f_n d\rho$ be a sequence of probability measures on $\T$ such that
\begin{enumerate}
  \item $\sigma_n$ converge weakly to some probability measure $\tau$.
  \item $f_n$ is positive almost everywhere with respect to $\tau$ and $\ds \Biggl( \bigintss \frac1{f_n} d\tau\Biggr)$ is a bounded
  sequence. Then
\end{enumerate}
$$\tau \perp \rho \Longleftrightarrow \lim_{n \longrightarrow +\infty}\bigintss \sqrt{f_n} d\rho =0.$$
\end{lem}
\begin{proof} Suppose $\tau$ and $\rho$ are mutually singular. Then, the affinity $G(\tau,\rho)$ is zero, which, by Coquet-Mand\'es-France-Kamae theorem (Theorem \ref{coquet-france}), gives
$$\lim_{n \longrightarrow +\infty}\bigintss \sqrt{f_n} d\rho=0.$$ Conversely, let $\varepsilon$ be a positive number, then there exists a large integer $n_0$ such that
$$\bigintss \sqrt{f_n} d\rho <\varepsilon.$$
 But, by our assumption (2), there exists $C>0$, such that,
 $$ \bigintss \frac1{f_n} d\tau<C.$$
Without loss of generality, let us assume  that $C=1$. Therefore, by Cauchy-Schwarz inequality, we have
$$\bigintss \frac1{\sqrt{f_n}} d\tau <1.$$
Hence,
$$\Biggl(\bigintss \sqrt{f_n} d\rho \Biggr) \Biggl(\bigintss \frac1{\sqrt{f_n}} d\tau\Biggr)<\varepsilon,$$
and this  implies that $\tau$ and $\rho$ are mutually singular, by the lemma below.
\end{proof}
The lemma is due to Kilmer and Saeki \cite{KilmerS}. We include the proof for the reader's convenience.
\begin{lem}[Kilmer-Saeki \cite{KilmerS}]Let $\rho$ and $\tau$ be a nonnegative finite measures on measurable space
$X$. Then the following properties are equivalent:
\begin{enumerate}[a)]
  \item $\rho \perp \tau.$
  \item Given $\varepsilon>0$, there exists nonnegative measurable function $f$ on $X$ such that
  $f>0,~\tau-$a.e. and such that
  \[
\Biggl(\bigintss f d\rho \Biggr) \Biggl(\bigintss \frac1{f} d\tau \Biggr) <\varepsilon.
  \]
\end{enumerate}
\end{lem}
\begin{proof} Suppose $a)$ obtains. Then, there exists two disjoint measurable sets $A,B$ such that
$\rho(A)=\rho(X)$ and $\tau(B)=\tau(X).$ For a given $\varepsilon>0,$ put
$$f=\varepsilon \1_A+\frac1{\varepsilon}\1_B.$$
We thus get
$$\Biggl(\bigintss f d\rho \Biggr) \Biggl(\bigintss \frac1{f} d\tau\Biggr)=\varepsilon \rho(X) \varepsilon \tau(X) <\varepsilon,$$
and this establish $b)$.\\
  Conversely suppose $b)$ obtains. Let $\tau'$ be a large measure such that $\tau' \leq \rho$ and $\tau'\leq \tau$. Given
  $\varepsilon>0$, let $f$ be a function furnished by $b)$. Since $f>0$, $\tau$-a.e., we also have
  $f>0$, $\tau'-$a.e.. Therefore, Cauchy-Schwarz inequality combined with $b)$ yields
  \begin{eqnarray*}
  \tau'(X)&=&\bigintss \sqrt{f} \frac1{\sqrt{f}} d\tau' \\
  &\leq&
\Biggl(\bigintss fd\tau' \Biggr)^\frac12 \Biggl(\bigintss \frac1{f} d\tau'\Biggr)^\frac12 \\
&\leq&  \Biggl(\bigintss fd\rho \Biggr)^\frac12 \Biggl(\bigintss \frac1{f} d\tau\Biggr)^\frac12\\
&\leq& \sqrt{\varepsilon}.
  \end{eqnarray*}
Since $\varepsilon>0$ was arbitrary, this gives that $\tau'=0$, which means that, $\rho \perp \tau.$ The proof of the lemma is complete.
\end{proof}
Let us recall the following important and classical fact from Probability Theory connected to the notion
of the uniform integrability.
\begin{Def}Let $(X,\cb,\Pro)$ be a probability space and $p \in [1,+\infty[$. A sequence $\{f_n, n\geq 1\}$ in $L^p(X)$ is said to be $L^p(X)$ uniformly integrable if
\[
\lim_{c \longrightarrow +\infty}\sup_{n \in \N}\bigintss_{\big\{|f_n| >c\big \}} \big|f_n\big |^p~~  d\Pro=0.
\]
\end{Def}
\noindent{}It is well-known that if
\begin{eqnarray}
\sup_{ n \in \N}\Big|\Big|f_n\Big|\Big|_{p+\varepsilon} < +\infty,
\end{eqnarray}
for some $\varepsilon$ positive, then the sequence $\{f_n\}$ is $L^p$ uniformly integrable.\\

\noindent{}It is obvious that the almost everywhere convergence does not in general imply
the convergence in $L^p(X)$. Nevertheless, it is well known that the condition of domination insure such convergence (Lebesgue's Dominated Convergence Theorem) but in the absence of domination the following Vital Convergence Theorem allows us to obtain the convergence in $L^p(X)$ provided that the sequence is uniformly integrability.
\begin{Th}[Vitali Convergence Theorem ] \label{Vitali}Let $(X,\cb,\Pro)$ be a probability space and $\{f_n\}$ be a uniformly integrable sequence in $L^p(X)$ which converges almost surely to some function $f$. Then $f$ is in $L^p(X)$ and
\[
\Big|\Big|f_n-f\Big|\Big|_p \tend{n}{+\infty}0.
\]
\end{Th}
\noindent{}For the proof of Theorem \ref{Vitali} we refer the reader to \cite[pp.134-135]{Rudin}, \cite[p. 165-167]{Schwartz} .\\

\section{Proof the main results}\label{S-Proofs}
Let $\sigma$ and $\tau$ be two measures on the circle. Then, by Lebesgue decomposition of  $\sigma$ with respect to $\tau$, we have
\[
\sigma=\frac{d\sigma}{d\tau} d\tau+\sigma_s,
\]
where $\sigma_s$ is singular to $\tau$ and $\ds\frac{d\sigma}{d\tau}$ is the Radon-Nikodym derivative. In the case of generalized Riesz product $\mu$, we are able to establish  that the following.

\begin{Prop}\label{Mcgheeps}Let $\mathcal{N}=\big\{k_0 <k_1 < k_2<\cdots\big\}$ be a subsequence and
$\nu=\prod_{ k \in \mathcal{N} } |P_k|^2$. Then, we have
\begin{eqnarray}
\bigintss \sqrt{\frac{d\nu}{dz}}dz  \leq \Bigg(\bigintss \Big|P_{k_0}(z)\Big| dz\Bigg)^{\frac12}.
\end{eqnarray}
\end{Prop}

\noindent{}In \cite{Bourisr}, J. Bourgain proved the following proposition.

\begin{Prop}[Bourgain \cite{Bourisr}]\label{BoMcps}Let $n$ be a positive integer and $~0\leq l_1<l_2<\cdots<l_n$ be a finite sequence of non negative integers. Then, there exists an absolute constant $c>0$ such that
\[
\bigintss \Big|\frac1{\sqrt{n}} \sum_{j=1}^{n} z^{l_i}\Big| dz \leq 1-c \frac{\log(n)}{n}.
\]
\end{Prop}
\noindent{}Proposition \ref{Mcgheeps} combined with Proposition \ref{BoMcps} yields the following
\begin{Cor}Let $\mathcal{N}=\big\{k_0 <k_1 < k_2<\cdots\big\}$ be a subsequence and
$\nu=\prod_{ k \in \mathcal{N} } |P_k|^2$. Then,
\begin{eqnarray}
\Biggl(\bigintss \sqrt{\frac{d\nu}{dz}} dz \Biggr)\leq \Biggl(1-c\frac{\log(m_{k_0})}{m_{k_0}}\Biggr)^{\frac12},
\end{eqnarray}
for some absolute constant $c>0$.
\end{Cor}
\begin{proof}[\bf Proof of Proposition \ref{Mcgheeps}] Let $N$ be a positive integer and $h$ be a continuous positive function. Then, By Cauchy-Schwarz inequality, we have
\begin{eqnarray*}
\Biggl(\bigintss \prod_{\overset{k \in \mathcal{N},} {k \leq N}} h \Big| P_k(z)\Big| dz\Biggr)^2
&\leq& \Biggl(\bigintss h\big|P_{k_0}(z)\big| dz\Biggr) \Biggl(\bigintss h \big|P_{k_0}(z)\big| \prod_{ \overset{k \in \mathcal{N}\setminus\{k_0\},}
{k \leq N}} \big|P_k(z)\big|^2 dz\Biggr).
\end{eqnarray*}
By letting $N$ goes to infinity, from Proposition \ref{weakly}, we deduce that
\begin{eqnarray}
\Biggl(\bigintss  h \sqrt{\frac{d\nu}{dz}}~dz\Biggr)^2
&\leq& \Biggl(\bigintss h\Big|P_{k_0}(z)\Big| dz\Biggr) \Biggl(\bigintss h d\nu\Biggr)^{\frac12} \Biggl(\bigintss h \frac{d\nu}{|P_{k_0}(z)|^2}\Biggr)^{\frac12}.
\end{eqnarray}
Since
\begin{align}
\bigintss h \Big|P_{k_0}(z)\Big| \prod_{ \overset{k \in \mathcal{N}\setminus\{k_0\},}
{k \leq N}} \Big|P_k(z)\Big|^2 dz=\bigintss h \prod_{\overset{k \in \mathcal{N},}
{k \leq N}} \Big|P_k(z)\Big|  \prod_{ \overset{k \in \mathcal{N}\setminus\{k_0\},}
{k \leq N}} \Big|P_k(z)\Big| dz  \nonumber \\
\leq \Biggl(\bigintss h \prod_{\overset{k \in \mathcal{N},}
{k \leq N}} \Big|P_k(z)\Big|^2  dz\Biggr)^{\frac12}  \Biggl( \bigintss h\prod_{ \overset{k \in \mathcal{N}\setminus\{k_0\},}
{k \leq N}} \Big|P_k(z)\Big|^2 dz\Biggr)^{\frac12},
\end{align}
Whence
\[
\Biggl(\bigintss \sqrt{\frac{d\nu}{dz}} dz \Biggr)^{2} \leq \bigintss \Big|P_{k_0}(z)\Big| dz, 
\]
and this proved the proposition.
\end{proof}
\noindent{}From Proposition \ref{Mcgheeps} combined with Lemma \ref{product}, we have the following lemma.
\begin{lem}
For any $k \in \N^*$, we have
\begin{eqnarray}
\bigintss \sqrt[2k]{\frac{d\mu}{dz}} dz \leq \prod_{j=0}^{k-1}\Big(1-c\frac{\log(m_j)}{m_j}\Big)^\frac{1}{2k}.
\end{eqnarray}
\end{lem}
\begin{proof}
Let $n_0$ be a positive integer and $\mathcal{N}_j$, $j=1,\cdots,n_0$ be a partition of $\N$. Denoted by $\mu_j$ the generalized
Riesz product construct over the subsequence $\mathcal{N}_j$ for each $j=1,2,\cdots,n_0$. Hence, by Lemma \ref{product}, we deduce that
 $d\mu/dz$ is the product of  $d\mu_i/dz$, $j=1,2,\cdots, n_0$. Indeed, one may take the Euclidian partition given by classifying the integers modulo $k$. Hence, by Lemma \ref{product}, we have
$$\sqrt{\frac{d\mu}{dz}} =\prod_{i=1}^{k} \sqrt{\frac{d\mu_i}{dz}}.$$
Consequently, by Proposition \ref{Mcgheeps} we get
\begin{eqnarray}
\Biggl(\bigintss \sqrt{\frac{d\mu_i}{dz}} dz \Biggr)^2\leq \bigintss |P_i| dz, {\textrm{~for~}} i=0,1,\cdots,k-1.
\end{eqnarray}
From this and  H\"{o}lder inequality, we conclude that,
\begin{eqnarray}
\bigintss \sqrt[2k]{\frac{d\mu}{dz}} dz \leq \prod_{j=0}^{k-1}\Big(1-c\frac{\log(m_j)}{m_j}\Big)^\frac{1}{2k}.
\end{eqnarray}
which proved the lemma.
\end{proof}
\noindent{}Now, let us proved Theorem  \ref{main2}.
\begin{proof}[\textbf{Proof of Theorem \ref{main2}}]It is easy to check from \eqref{mahlerzer0} that, for any $k \geq 1$,
 \begin{eqnarray}
M\Big(\frac{d\mu}{dz}\Big)\leq \prod_{j=0}^{k-1}||P_j||_1.
\end{eqnarray}
This combined with Proposition \ref{BoMcps}, yields
\[
M\Biggl(\frac{d\mu}{dz}\Biggr)=\prod_{j=0}^{k-1}M\Big(\sqrt{\frac{d\mu_j}{dz}}\Big)\leq \prod_{j=0}^{k-1}\Big(1-c\frac{log(m_j)}{m_j}\Big).
\]
Taking into account that $m_j=\theta(j^\beta)$, with $\beta \leq 1$, it follows that
\[
M\Biggl(\frac{d\mu}{dz}\Biggr) \leq \prod_{j=0}^{k-1}\Big(1-c\frac{log(m_j)}{m_j}\Big) \tend{j}{+\infty}0,
\]
which proves the theorem.
\end{proof}
At this point let us present the proof of Theorem \ref{nad3}.
\begin{proof}[\textbf{Proof of Theorem \ref{nad3}}]By Proposition \ref{Mcgheeps} we have
$$\bigintss \sqrt{\frac{d\mu}{dz}} dz \leq \Bigg(\bigintss \Big|P_{m_0}(z)\Big| dz\Bigg)^{\frac12}.$$
Moreover, by Cauchy-Schwarz inequality, we can write 
$$\bigintss \Big|P_{m_0}(z)\Big| dz<1,$$
Since the equality in Cauchy-Schwarz inequality holds if and only if $\Big|P_{m_0}(z)\Big|$ is a constant polynomial which is impossible in our case. The proof of the theorem is complete.
\end{proof}
\begin{rem}The proof of Theorem \ref{nad3} can be obtained also by combining Propositions \ref{Mcgheeps} and \ref{BoMcps}. 
\end{rem}

We shall need the following lemma
\begin{lem}\label{MahlerConverge}
Let $\tau$ be a finite measure on a Borel space $X$. Suppose that the sequences of positive functions $(f_n)_{n \geq 0}$ and
$\ds \frac1{f_n}$  in $L^1(\tau)$ converge in $L^1(\tau) $ to
$1$. Then, there is a subsequence $(f_{n_i})_{i \geq 0}$ such that the Mahler measure of $f_{n_i}$  converge to 1, that is,
\[
\exp\Biggl(\bigintss \log\Bigl(f_{n_i}\Bigr) d\tau\Biggr) \tend{i}{+\infty}1.
\]
\end{lem}
\begin{proof}By our assumption there is a subsequence $(f_{n_i})_{i \geq 0}$ such that $f_{n_i}$ and $\ds \frac1{f_{n_i}}$ converge almost everywhere to $1$. Therefore,
\[
\log^{+}\big(f_{n_i}\big) \tend{i}{\infty} 0  {\textrm {~~and~~}}  \log^{-}\big(f_{n_i}\big) \tend{i}{\infty} 0,
\]
\noindent{}where $\log^{+}(t)=\max(0,log(t))$ and $\log^{-}(t)=\max(0,-log(t))$. We claim that the sequence
$\log(f_{n_i})_{i \geq 0}$ converge to $0$ in $L^1(\lambda)$. Indeed,
the sequences  $\log^{+}\big(f_{n_i}\big)$ and $\log^{-}\big(f_{n_i}\big)$ are uniformly integrable. Since, for any positive number $C$  we have
\begin{eqnarray*}
\Big\{\log^{+}\big(f_{n_i}\big)>C\Big\}&=&\Big\{\log^{+}\big(f_{n_i}\big)>C, f_{n_i} \leq 1 \Big\} \bigcup
\Big\{\log^{+}\big(f_{n_i}\big)>C, f_{n_i} > 1 \Big\}\\
&=&\Big\{f_{n_i}>e^C, f_{n_i} > 1 \Big\}=\Big\{f_{n_i}>e^C \Big\}.
\end{eqnarray*}
\noindent{}Hence,
\begin{eqnarray*}
\frac12\bigintss_{\big\{ \log^{+}\big(f_{n_i}\big) >C\big\}} \log^{+}\big(f_{n_i}\big) d\tau
&=&
\bigintss_{\big\{ f_{n_i} >e^C\big\}} \log\big(\sqrt{f_{n_i}}\big) d\tau\\
&\leq&
\bigintss_{\big\{ f_{n_i} >e^C\big\}} \sqrt{f_{n_i}} d\tau,
\end{eqnarray*}
since $\log(x) \leq x$ for any $x \geq 1$. This combined with Cauchy-Schwarz inequality yields
\begin{eqnarray*}
\frac12\bigintss_{\big\{ \log^{+}\big(f_{n_i}\big) >C\big\}} \log^{+}\big(f_{n_i}\big) d\tau
&\leq& \tau\Biggl(\Bigl\{ f_{n_i} >e^C\Bigr\}\Biggr)^{\frac12} \Biggl(\bigintss f_{n_i} d\tau \Biggr)^{\frac12},
\end{eqnarray*}
and, by Markov inequality, we get
\begin{eqnarray*}
\frac12\bigintss_{\big\{ \log^{+}\big(f_{n_i}\big) >C\big\}} \log^{+}\big(f_{n_i}\big) d\tau
&\leq& e^{-C/2} \sup_{i \geq 0}||f_{n_i}||_1\\
&\leq& e^{-C/2} K \tend{C}{+\infty} 0,
\end{eqnarray*}
\noindent{}where $K$ is some positive constant such that $\ds \sup_{i \geq 0}||f_{n_i}||_1<K$. It is remains to prove that
$\log^{-}\big(f_{n_0}\big)$ is uniformly integrable. Notice that
\begin{eqnarray}
\Big\{\log^{-}\big(f_{n_i}\big)>C\Big\}&=&\Big\{\log^{-}\big(f_{n_i}\big)>C, f_{n_i} < 1 \Big\} \bigcup
\Big\{\log^{-}\big(f_{n_i}\big)>C, f_{n_i} \geq 1 \Big\} \nonumber\\
&=&\Big\{{\frac1{f_{n_i}}}>e^C, f_{n_i} <1 \Big\}= \Big\{{\frac1{f_{n_i}}}>e^C \Big\}.
\end{eqnarray}
\noindent{}Therefore, in the same manner as before, we obtain
\begin{eqnarray*}
\bigintss_{\big\{ \log^{-}\big(f_{n_i}\big) >C\big\}} \log^{-}\big(f_{n_i}\big) d\tau
&\leq& e^{-C/2} K' \tend{C}{+\infty} 0,
\end{eqnarray*}
\noindent{}where $K'$ is some positive constant such that $\ds \sup_{i \geq 0}\Bigl|\Bigl|\frac1{f_{n_i}}\Bigr|\Bigr|_1<K'$.  Thus, by Vitali Convergence Theorem, it follows that $\log^{+}\big(f_{n_i}\big)$ and $\log^{-}\big(f_{n_i}\big)$ converge to $0$ in $L^1(\tau)$, and this gives
\[
\bigintss \Bigr|\log\big(f_{n_i}\big)\Bigl| d\tau \tend{i}{\infty} 0,
\]
which yields that the Mahler measure of the sequence $(f_{n_i})$  converge to 1, that is,
\[
M_{\tau}(f_{n_i}) =\exp\Biggl(\bigintss \log\Bigl(f_{n_i}\Bigr) d\tau\Biggr) \tend{i}{\infty}1,
\]
 and the proof of the lemma is achieved.
\end{proof}

\begin{proof}[\textbf{Proof of Theorem \ref{dicho}}]Put
$$e=\prod_{j \geq 0 }M(P_j) {\textrm{~~,~~}}Q_n=\prod_{j=0}^{n}|P_j|,
f=\sqrt{\frac{d\mu}{dz}}  {\textrm{~~and~~}} B=\Big\{f^2>0\Big\}.$$
Suppose $e$ is positive. Then,
by Proposition \ref{basic}, it follows that
$$\lambda(B)=1,$$
which means that the absolutely continuous part of $\mu$ is equivalent to Lebesgue measure $\lambda$. Conversely, assume that $\lambda$ is equivalent to $\frac{d\mu}{d\lambda}d\lambda$ then for any non-negative integer $n$,  $\lambda$ is equivalent to $\frac{d\mu_n}{d\lambda}d\lambda$, and  by Theorem \ref{affinity},  we can write
$$\sqrt{\frac{d\mu_n}{d\lambda}}=\prod_{j=n}^{+\infty}|P_j|,$$
\noindent{}in the sense of $L^1(dz)$. Put

 $$\phi_n=\sqrt{\frac{d\mu_n}{d\lambda}} \; \; \textrm{and} \;\;\nu_{n}=\min(1,\phi_n^2)~ dz,~~n \in \N.$$
It follows that
for each $n \in \N$,  $\nu_{n} \leq \mu_n$ and $\nu_n \leq \lambda$ and,
there is a subsequence $(\prod_{j=0}^{N_k}|P_j|)$ which converge almost everywhere to $\phi_0$ and $\phi_0$ is positive almost everywhere. Hence, by Cauchy criterium, for almost all $z \in \T$,
\[
\prod_{k=n}^{+\infty}|P_k(z)| \tend{n}{\infty}1,
\]
that is, for almost all $z \in \T$,
\begin{eqnarray}\label{Cauchy}
\phi_n(z) \tend{n}{\infty}1,
\end{eqnarray}
and since $(\phi_n(z))_{n \geq 0}$ is bounded in $L^2(dz)$, we have that $(\phi_n(z))_{n \geq 0}$ is uniformly integrable in
$L^1(dz)$, which implies, by Vitali Convergence Theorem (Theorem \ref{Vitali}), that $(\phi_n(z))_{n \geq 0}$ converge in
$L^1(dz)$ to $1$. Obviously, from \eqref{Cauchy}, for almost all $z \in \T$,
 \begin{eqnarray}\label{Cauchyinv}
{\frac1{\phi_n(z)}} \tend{n}{\infty}1,
\end{eqnarray}
again we have  that $\ds \Bigl({\frac1{\phi_n(z)}}\Bigr)_{n \geq 0}$ is uniformly integrable in
$L^1(\phi^2_0 dz)=L^1(d\mu)$. Indeed,
\[
\bigintss {\frac1{\phi^2_n(z)}} \phi^2_0 dz=\bigintss {\frac1{\phi^2_n(z)}} d\mu=
\bigintss {\frac1{\phi^2_n(z)}} Q_n^2 d\mu_{n}=\bigintss Q_n^2 dz=1.
\]
and this gives that $\ds \Bigl({\frac1{\phi_n(z)}}\Bigr)_{n \geq 0}$ converge to $1$ in $L^1(\phi^2_0 dz)$, hence
$\ds \Bigl({\frac1{\phi_n(z)}}\Bigr)_{n \geq 0}$ converge to $1$ in $L^1(\nu_0)$, since
\begin{eqnarray}
\bigintss \Bigg|{\frac1{{\phi_n(z)}}}-1\Bigg| d\nu_0
\leq \bigintss \Bigg|{\frac1{\phi_n(z)}}-1\Bigg| d\mu.
\end{eqnarray}
\noindent{}We thus get that the sequences $(\phi_n(z))_{n \geq 0}$ and $\ds
\left(\frac1{\phi_n(z)}\right)_{n \geq 0}$ converge to $1$ in $L^1(\nu_0)$. Whence, by virtue of Lemma \ref{MahlerConverge},
the Mahler measure of a subsequence $({\phi_{n_i}(z)})_{i \geq 0}$ converge to $1$, and by Theorem \ref{AN}, we have
\[
\prod_{j=n_i}^{+\infty}M(P_j) \tend{i}{+\infty}1.
\]
We thus get that $e$ is positive, since for any positive integer $n$, we have
\[
M\Biggl(\frac{d\mu}{dz}\Biggr)=\prod_{0}^{n-1}M(P_j).M\Biggl(\frac{d\mu_n}{dz}\Biggr).
\]
It is remains to prove the second part of the theorem. Suppose $\mu$ and $\lambda$ are mutually singular. Then, by Lemma \ref{orth}, we have
$$\lim_{N} \bigintss \prod_{j=0}^{N}\Big|P_j(z)\Big| =0,$$
which, by Proposition \ref{basic}, gives
$$\lim_{N}  \prod_{j=0}^{N}M(P_j) =0=e.$$  For the converse, suppose that $e=0$ and $\mu \not \perp \lambda$. Then, there exists
a positive measure $\tau$ such that $\tau \leq \mu$ and $\tau \leq \lambda$. Notice that we further have
\[
M\Biggl(\frac{d\mu_n}{dz}\Biggr)=0, {\textrm{~~for~~all~~}} n\in \N, \eqno (\star)
\]
by Theorem \ref{AN}.

We shall apply the same reasoning as before. Since $Q_n$ converge to $\phi_0$ in $L^1(dz)$, it follows easily that $Q_n$ converge to $\phi_0$ in $L^1(\tau)$ and there is a subsequence $Q_{n_i}$ such that, for almost all $z \in \T$ (with respect to
$\tau$),
\[
\prod_{j=0}^{n_i-1}\Bigl|P_j(z)\Bigr| \tend{i}{+\infty}\phi_0(z).
\]
Hence, for almost all $z \in \T$ (with respect to
$\tau$)
\[
\phi_{n_i} \setdef \prod_{j=n_i}^{+\infty}\Bigl|P_j(z)\Bigr| \tend{i}{+\infty}1.
\]
\noindent{}Moreover the sequences $\Bigl(\phi_{n_i}\Bigr)$ and
$\Bigl(\ds {\frac1{\phi_{n_i}}}\Bigr)$ are uniformly integrable in $L^1(\tau)$, since $\tau \leq \lambda$ and $\tau \leq \mu$. Therefore, by Vitali Convergence Theorem,
$\Bigl(\phi_{n_i}\Bigr)$ and $\Bigl(\ds {\frac1{\phi_{n_i}}}\Bigr)$ converge in $L^1(\tau)$ to $1$. Hence, by Lemma \ref{MahlerConverge}, we have
\[
\exp\Biggr(\bigintss \log\Bigr(\prod_{j=n_i}^{+\infty}\bigl|P_j(z)\bigr|\Bigl) d\tau\Biggl) \tend{i}{+\infty}1.
\]
Notice that the careful application of Lemma \ref{MahlerConverge} yields that
$$\bigintss \log^{+}\Bigr(\prod_{j=n_i}^{+\infty}\bigl|P_j(z)\bigr|\Bigl) d\tau,$$
and this gives that
$$\bigintss \log^{+}\Bigr(\prod_{j=0}^{+\infty}\bigl|P_j(z)\bigr|\Bigl) d\tau<+\infty.$$
We may also take $\tau=\prod_{j=0}^{+\infty}\bigl|P_j(z)| dz$ and repeated the same reasoning to get
$$\bigintss \sqrt{\frac{d\mu}{dz}}\log^{+}\Bigr(\sqrt{\frac{d\mu}{dz}}\Bigl) d\tau<+\infty.$$
We thus get
$$F(z)=\bigintss \frac{e^{it}+z}{e^{it}-z} \sqrt{\frac{d\mu}{dz}}(t) dt \in \Ha_1,$$
by $L log L$ Zygmund theorem (Theorem \ref{zyg}) . This forces $\sqrt{\frac{d\mu}{dz}}$ to be a non zero constant function a.e., and yields that
\[
\prod_{j=0}^{+\infty}\exp\Biggr(\bigintss \log\Bigr(\bigl|P_j(z)\bigr|\Bigl) dz\Biggl)=e>0,
\]
which is impossible in view of ($\star$) combined with Theorem \ref{AN}. Thus $\tau \equiv 0$, that is,  $\mu \perp \lambda,$  and this finishes the proof of the theorem.
\end{proof}

\begin{rem}  1) By a standard argument from spectral analysis \cite[pp. 17]{Queffelec1},
\begin{eqnarray*}
M\Bigl(\frac{d\mu}{dz}\Bigr)&=&\inf_{P \in A_0}\bigintss d\sigma_{(1-P)(U_T)\xi_0}\\
&=&\inf_{P \in A_0}\Bigl\|\xi_0-\bigl(P(U_T)(\xi_0)\bigr)\Bigr\|_2\\
&=&\inf_{\chi \in H_{-1}}\Bigl\|\xi_0-\chi\Bigr\|_2,
\end{eqnarray*}
where $H_{-1}$ is the closed subspace generated by $\big\{T^k \xi_0, k \leq -1\big\}$. Thus
,by Fenchel-Rockafellar duality theorem \cite[pp. 15-17]{Brezis},  we have
$$M\Bigl(\frac{d\mu}{dz}\Bigr)=\sup_{\overset{\chi \in {{H_{-1}}^{\perp}}}{||\chi||_2 \leq 1}}\Biggl|\bigintss \chi \xi_0 d\Pro\Biggr|,$$
where $H_{-1}^{\perp}$ is the orthogonal complement of $H_{-1}$. Assume that we have
$M\Bigl(\frac{d\mu}{dz}\Bigr)=0$. Then, Our proof yields that the spectral type of the associated rank one is singular, we thus have $H_{-1}=L^2(X)$, where $H_{-1}$ is the closed subspace generated by $\big\{T^k \xi_0, k \leq -1\big\}$ and
$$\sup_{\overset{\chi \in {{H_{-1}}^{\perp}}}{||\chi||_2 \leq 1}}\Biggl|\bigintss \chi \xi_0 d\Pro\Biggr|=0.$$
2) Notice that it is easy to see that for any $\chi \in L^2(X)$,
$$\bigintss \chi \xi_n d\Pro \tend{n}{+\infty}0.$$
Indeed, assume that $\chi$ is the indicator function of some set, then
$$\bigintss \chi \xi_n d\Pro  \leq \sqrt{|B_n|} \tend{n}{+\infty}0.$$
Now, write
$$\chi=\sum_{j=1}^{m}\chi_j \1_{A_j},$$
then
$$\Biggl|\bigintss \chi \xi_n d\Pro  \Biggr| \leq
\Bigl(\sum_{j=1}^{m}|\chi_j| \Bigr)\sqrt{|B_n|}  \tend{n}{+\infty}0.$$
We conclude by the density of the simple functions in $L^2(X)$. An alternative proof can be obtained with the help of Lebesgue density theorem.
\end{rem}

 From Theorem \ref{dicho} the proof of Theorem \ref{m1} is straightforward. Indeed,
\begin{proof}[{\textbf{Proof of Theorem \ref{m1}}}]By Theorem \ref{dicho}, the spectral type of any rank one is either singular or its absolute continuous part is equivalent to the Lebesgue measure according as the product of the Mahler measure of $P_j$ converge or diverge. Precisely, it is singular if and only if
the product diverge otherwise the absolute continuous part is equivalent to the Lebesgue measure. The proof of the theorem is complete.
\end{proof}

\noindent{}Let us prove Corollaries \ref{singular} and  \ref{CAN}.

\begin{proof}[\textbf{Proof of Corollaries \ref{singular} and \ref{CAN}.}]
Straightforward, by  Theorem \ref{dicho} combined with Theorem \ref{main2}.
\end{proof}

\begin{thank}
The author wishes to express his heartfelt thanks to Mahendra Nadkarni who inspired largely this work and for his encouragements and many stimulating conversations. He would like also to express his thanks to Fran\c cois Parreau who introduced him to the Riesz products business and to Jean-Paul Thouvenot who introduced him to the wild world of rank one maps.
\end{thank}

\end{document}